\documentclass{amsart}

\pdfoutput=1

\usepackage[margin=1.0in]{geometry}
\usepackage{amssymb}
\usepackage{amsmath,amsthm}
\usepackage{graphicx}
\usepackage{tikz,lipsum,lmodern}
\usepackage{caption} 
\usetikzlibrary{patterns}
\usetikzlibrary { decorations.pathmorphing, decorations.pathreplacing, decorations.shapes, }
\usepackage{graphicx}
\usepackage{ytableau}
\usepackage{float}
\usepackage{hyperref}
\usepackage{tikz-qtree}
\usetikzlibrary{shadows,trees}
\usepackage{colortbl}
\usepackage{color}
\usepackage{url}
\usepackage{amsfonts}
\usepackage{amsrefs}
\usepackage{multicol}
\usepackage{array}
\usepackage{booktabs}
 \usepackage{cleveref}
\numberwithin{equation}{section}
\numberwithin{figure}{section}
\numberwithin{table}{section}

\newtheorem{theorem}{Theorem}[section]
\newtheorem{corollary}[theorem]{Corollary}
\newtheorem{proposition}[theorem]{Proposition}

\newtheorem{definition}[theorem]{Definition}
\newtheorem{remark}[theorem]{Remark}
\newtheorem{example}[theorem]{Example}
\newtheorem{problem}[theorem]{Problem}
\allowdisplaybreaks

\begin{document}

\title{Palindrome Partitions and the Calkin-Wilf Tree}
\author{David J. Hemmer}
\address{Department of Mathematical Sciences\\
Michigan Technological University\\ Houghton, MI 49931} \email{djhemmer@mtu.edu}

\author{Karlee J. Westrem}
\address{Department of Mathematical Sciences\\
Michigan Technological University\\ Houghton, MI 49931} \email{kjwestre@mtu.edu}
\date{May, 2024}
\begin{abstract}
There is a well-known bijection between finite binary sequences and integer partitions. Sequences of length $r$ correspond to partitions of perimeter $r+1$. Motivated by work on rational numbers in the Calkin-Wilf tree, we classify partitions whose corresponding binary sequence is a palindrome. We give a generating function that counts these partitions, and describe how to efficiently generate all of them. Atypically for partition generating functions, we find an unusual significance to prime degrees. Specifically, we prove there are nontrivial \textit{palindrome partitions} of $n$ except when $n=3$ or $n+1$ is prime.  We find an interesting new ``branching diagram" for partitions, similar to Young's lattice, with an action of the Klein four group corresponding to natural operations on the binary sequences.
\end{abstract}

\maketitle

\section{Introduction}
Recall that $\lambda=(\lambda_1, \lambda_2, \ldots, \lambda_r)$ is a partition of $n$, denoted $\lambda \vdash n$, if $\lambda_1 \geq \lambda_2 \geq \cdots \geq \lambda_r>0$ are positive integers with $\sum_{i=1}^{r}\lambda_i=n$. To $\lambda$ we associate its \emph{Young diagram} $[\lambda]$. It is well-known that we can encode partitions by infinite binary sequences, unbounded in each direction, which begin with infinitely many zeros and end with infinitely many ones, see for example \cite[Exercise 7.59]{StanleyEC2} or \cite{ArmstrongResultsandconjectures}. This encoding can be used, for example, to prove  the important fact in representation theory of the symmetric group that the $p$-core of a partition is unique. To encode a partition $\lambda$, draw the Young diagram of $\lambda$ with the left hand vertical edge and upper horizontal edge extended to infinity. Label the southeast boundary of the diagram, including the horizontal and vertical rays, with a zero next to each vertical edge and a one next to each horizontal edge. To obtain the corresponding binary sequence $C_\lambda$, simply read the numbers moving from southwest to northeast.  For example in Figure \ref{fig:codingofpartitions} we see that if $\lambda=(5,5,3,3,1)$ then $C_\lambda=\cdots 00101100110011\cdots.$

\begin{figure}[H]

\centering

   \begin{tikzpicture}[scale=0.6]

         \draw  (0,1) -- (0,6);
        \draw  (0,6) -- (5,6);
          \draw  (0,6) -- (8,6);
          \draw  (0,1) -- (0,-2);
        \draw  (5,4) -- (3,4);
        \draw  (3,4) -- (3,2);
         \draw  (3,2) -- (1,2);
          \draw  (1,2) -- (1,1);
        \draw (0,4) grid (5,6);
  \draw (0,2) grid (3,4);
    \draw (0,1) grid (1,2);
    \node at (0,0) [circle,fill,inner sep=1.0pt]{};
        \node at (0,-1) [circle,fill,inner sep=1.0pt]{};
            \node at (0,-1) [circle,fill,inner sep=1.0pt]{};
                        \node at (6,6) [circle,fill,inner sep=1.0pt]{};
                                    \node at (7,6) [circle,fill,inner sep=1.0pt]{};

    \node at (5.7,6.5){1};
        \node at (6.7,6.5){1};
            \node at (7.5,6.5){$\cdots$};
              \node at (5.3,5.5){0};
              \node[text=red] at (5.3,4.5){0};
              \node[text=red] at (5.3,4.5){0};
               \node[text=red] at (4.7,3.7){1};
                 \node[text=red] at (3.7,3.7){1};
                      \node[text=red] at (3.2,3.5){0};
                        \node[text=red] at (3.2,2.5){0};
                         \node[text=red] at (2.7,1.7){1};
                 \node[text=red] at (1.7,1.7){1};
                       \node[text=red] at (1.2,1.5){0};
        \node at (0.7,0.7){1};
    \node at (0.2,0.5){0};
                        \node at (0.2,-0.5){0};
                           \node at (0.2,-1.5){$\vdots$};

    \end{tikzpicture}

\caption{The coding $C_\lambda$ of the partition $\lambda=(5,5,3,3,1)$}
  \label{fig:codingofpartitions}
\end{figure}

After the infinitely many zeros, $C_\lambda$ has an initial one, corresponding to the first rightward move, and a final zero, corresponding to the last upwards move, before the infinite sequence of ones. 

\begin{definition}
\label{def: DefineBlambda}
  Define $B(\lambda)$ to be the finite binary sequence between the initial one and final zero in $C_\lambda$. So in Figure \ref{fig:codingofpartitions}, we see that $B(5,5,3,1)=01100110$, labelled in red.

\end{definition}
 Of course one can easily recover $\lambda$ from $B(\lambda)$. Given a binary sequence $C$ we denote the corresponding partition by $P(C)$, so $\lambda=P(B(\lambda))$. Notice that the partition $\tau=(1)$ corresponds to the empty sequence $B(\tau)=\emptyset$, so $P(\emptyset)=(1)$. There is no binary sequence corresponding to the ``empty partition."

 Suppose we have a sequence $C$ with $A$ zeros and $B$ ones. It is clear from Figure \ref{fig:codingofpartitions}, that $B+1$ is the first part of $P(C)$ and $A+1$ is the number of nonzero parts of $P(C)$.

 There are two natural operations on these binary sequences. The \emph{reverse} sequence just reverses the order in the sequence, and we denote it by $C^r$. The \emph{inverse} of a sequence is obtained by switching zeros and ones, and is denoted $\overline{C}$.

\begin{remark}
\label{remark:operationonpartitions}There is an equivalent way to define $B(\lambda)$ in terms of operations on the Young diagram. Start with the Young diagram of $\lambda$. If $[\lambda]$ has a single box in the last row (i.e. if the last part is equal to 1), then remove it and record a ``0". If not, then remove the entire first column from the Young diagram (i.e. subtract one from every part), and record a ``1". Continue until only one box remains. The sequence you obtain will be $B(\lambda)$.
\end{remark}

Reversing $B(\lambda)$ has a nice geometric description in terms of Young diagrams:

\begin{proposition}
  \label{prop: describereversal}
  Let $\lambda$, $B(\lambda)$, $A$ and $B$ be as above.  Consider the  boxes not in the first row or column as a partition $\tilde{\lambda}$ sitting inside an $A \times B$ rectangle. Replace $\tilde{\lambda}$ with the partition obtained by taking its complement inside the $A \times B$ rectangle and rotating it 180 degrees, while preserving the first row and column of $\lambda$. What remains will be the Young diagram for $P(B(\lambda)^r)$.
\end{proposition}

\begin{proof}
This is clear from Figure \ref{fig:reversalsequence}. When we rotate the rectangle by 180 degrees, vertical edges remain vertical, horizontal edges remain horizontal. So we can simply rotate all the labels (in red) and the new partition will be correctly labelled with binary sequence equal to $B(\lambda)^r$. 
\end{proof}

\begin{example}
  \label{ex:exreversal}
  Consider the sequence 010100. Check that for $\lambda=(3,3,3,2,1)$ we have $B(\lambda)=010100$. Following the algorithm in Prop. \ref{prop: describereversal}, we see that the partition $(3,3,2,1,1)$ corresponds to the reverse sequence 001010. This is illustrated  in Figure \ref{fig:reversalsequence}. Notice that $n$ is not typically preserved by this operation.
\end{example}

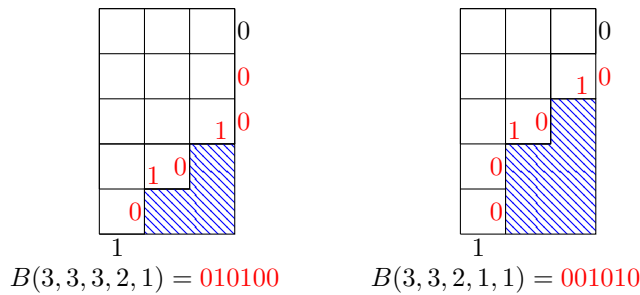
\begin{figure}[H]

\centering

   \begin{tikzpicture}[scale=0.6]

         \draw  (0,0) -- (0,5);
        \draw  (0,5) -- (3,5);
          \draw  (3,5) -- (3,2);
          \draw  (3,2) -- (2,2);
        \draw  (2,2) -- (2,1);
        \draw  (2,1) -- (1,1);
         \draw  (1,1) -- (1,0);
          \draw  (1,0) -- (0,0);
        \draw (0,2) grid (3,5);
  \draw (0,1) grid (2,2);
 \node at (3.2,4.5){0};
    \node[text=red] at (3.2,3.5){0};
    \node[text=red] at (3.2,2.5){0};
     \node[text=red] at (2.7,2.3){1};
     \node[text=red] at (1.8,1.5){0};
      \node[text=red] at (1.2,1.3){1};
   \node[text=red] at (0.8,0.5){0};
  \node at (0.4,-0.3){1};
  \draw[pattern=north west lines, pattern color=blue] (1,0) --(1,1)--(2,1)--(2,2)--(3,2)--(3,0)--(0,0);

     \node at (1,-1){$B(3,3,3,2,1)=\textcolor{red}{010100}$};

  \draw  (8,0) -- (8,5);
    \draw  (8,5) -- (11,5);
      \draw  (11,5) -- (11,4);
      \draw  (11,4) -- (10,4);
      \draw  (10,4) -- (10,2);
      \draw  (10,2) -- (9,2);
       \draw  (9,2) -- (9,0);
         \draw  (9,0) -- (8,0);
         \draw  (8,1) -- (9,1);
  \draw (8,2) grid (10,5);
  \draw (10,3) grid (11,5);

  \node at (11.2,4.5){0};
    \node[text=red] at (11.2,3.5){0};
     \node[text=red] at (10.7,3.3){1};
 \node[text=red] at (9.8,2.5){0};

       \node[text=red] at (9.2,2.3){1};
        \node[text=red] at (8.8,0.5){0};
         \node[text=red] at (8.8,1.5){0};
   \node at (8.4,-0.3){1};
  
    \draw[pattern=north west lines, pattern color=blue] (9,0)--(9,2)--(10,2)--(10,3)--(11,3)--(11,0)--(9,0);

     \node at (9,-1){$B(3,3,2,1,1)=\textcolor{red}{001010}$};

    \end{tikzpicture}

\caption{Calculating the partition corresponding to the reverse sequence}
  \label{fig:reversalsequence}
\end{figure}

\section{The Calkin-Wilf Tree}
We are going to be studying properties of the sequence $B(\lambda)$, and in particular when it is a palindrome. This question was motivated by work of \cite{KenyonCalkinWilfThesis} on binary sequences and the Calkin-Wilf tree. This is an infinite perfect binary tree with each vertex labelled by a rational number. One can also label this tree by binary sequences in a natural way, and also by Young diagrams of partitions. We give a quick description in this section as motivation, but it is not necessary to understand our results.

The Calkin-Wilf tree, $T_\mathbb{Q}$, and corresponding Calkin-Wilf sequence, is defined in \cite{CalkinWilfRecountingTheRationals}, although it has appeared also much earlier. It is an infinite binary tree with each vertex labelled by a positive rational number written in lowest terms. It is defined inductively as follows. The root is $\frac{1}{1}$. The left child of $\frac{p}{q}$ is $\frac{p}{p+q}$ and the right child is $\frac{p+q}{q}$. Part of the tree is shown on the right side of Figure \ref{fig:BinaryTree}. It is not difficult to prove that every positive rational number appears exactly once, and that if $\frac{p}{q}$ is a reduced fraction, then so are both of its children. To obtain the Calkin-Wilf sequence, read left to right top to bottom, obtaining $\{1, \frac{1}{2}, \frac{2}{1}, \frac{1}{3}, \frac{3}{2}, \frac{2}{3}, \frac{3}{1}, \ldots\}$. Let $l(n)$ be the $n$-th term in the Calkin-Wilkin sequence. In \cite{GraverListingthepositiverationals}, the author gives a method to calculate $l(n)$ without generating the entire tree above it. In the opposite direction, in \cite{GoblerListingTheRationals} the author gives a method to start with a positive rational number and locate its position in the sequence, and hence on the tree. One first computes the simple continued fraction expansion of the rational number, and then uses it to read off the binary sequence. For example starting with $\frac{9}{7}$, one obtains:
$$\frac{9}{7}=\textcolor{red}{1}+\frac{1}{\textcolor{red}{3}+\frac{1}{\textcolor{red}{2}}}.$$ The sequence $\{2,3,1\}$ encodes two ``1"s, three ``0"s, one ``1", thus the binary sequence $110001$.

In Figure \ref{fig:BinaryTree} we give, on the left, the infinite binary tree labeled by binary sequences starting with 1. So the root is labelled 1, the left child corresponds to appending a zero and the right child to appending a one to the end of the sequence.

\tikzset{font=\small,
level distance=1.15cm,
every node/.style=
    {align=center,
    text depth = 0pt
    },edge from parent/.style=
    {draw=black!50,
    thick
    }}
\begin{figure}[H]
\centering
\begin{tikzpicture}[scale=1.0]
\Tree [.1
        [.{10}
            [.{100}
                [.{1000} ]
                [.{1\textcolor{red}{001}} ]
            ]
            [.{101}
            [.{1010} ]
                [.{1011} ]
            ] ]
        [.11
            [.{110}
                [.{1100} ]
                [.{1101} ]
                 ]
            [.{111}
                [.{1110} ]
                [.{1111} ]
                ]
            ]
                        ]
]

\end{tikzpicture}\hspace{1cm}
\begin{tikzpicture}[scale=1.0]
\Tree [.1/1
        [.{$1/2$}
            [.{$1/3$}
                [.{$1/4$} ]
                [.{\textcolor{red}{$4/3$}} ]
            ]
            [.{$3/2$}
            [.{$3/5$} ]
                [.{$5/2$} ]
            ] ]
        [.$2/1$
            [.{$2/3$}
                [.{$2/5$} ]
                [.{$5/3$} ]
                 ]
            [.{$3/1$}
                [.{$3/4$} ]
                [.{$4/1$} ]
                ]
            ]
                        ]
]

\end{tikzpicture}
\caption{The Binary Tree and Calkin-Wilf tree}
  \label{fig:BinaryTree}
\end{figure}
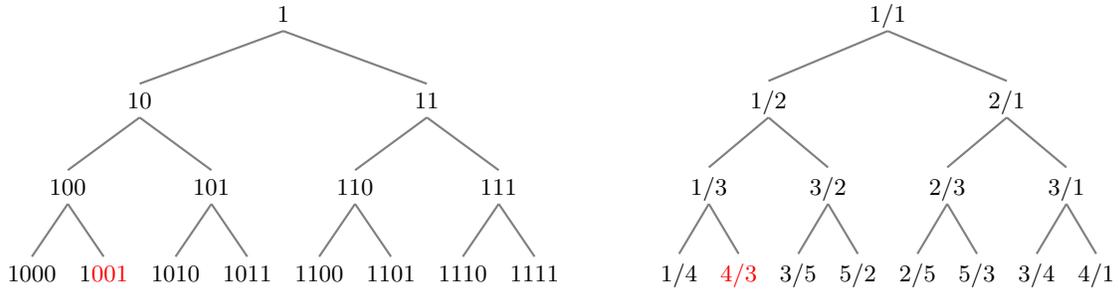

Ignoring the initial 1 on each label, we have a bijection between binary sequences of length $n$ and entries in the tree at level $n+1$. For example the binary sequence 001 corresponds to the fraction $4/3$, as labelled in red.

In \cite{KenyonCalkinWilfThesis}, Kenyon studies the correspondence between  binary sequences and rational numbers illustrated in Figure \ref{fig:BinaryTree}. She calls the sequences ``paths," as they encode the path from the root vertex to the vertex in question, with zero representing left and one representing right. Again note that she does not consider the initial one as part of the path, the rational number $\frac{1}{1}$ corresponds to an empty path.

 It is an easy exercise that the inverse paths (sequences) corresponds to reciprocal fractions. It is not at all clear what reversing the path does to the corresponding rational numbers.

Kenyon gathers data on paths that are equal to their reverse, called \emph{palindromic paths}, and those whose reverse path is equal to the inverse path, called \emph{antipalindromic paths}. So, for example, $0110$ is a palindromic path and corresponds to the rational number $5/7$. The path $0101$ is an antipalindromic path, corresponding to the rational number $8/5$.

We emphasize that the tree on the left of Figure \ref{fig:BinaryTree} is labelled in a way to match the previous literature, but the reader should keep in mind when we discuss paths, palindromic paths, etc. that the initial one is not included.

\subsection{Partitions in the Calkin-Wilf Tree}
From Remark \ref{remark:operationonpartitions} we see there is a natural way to place Young diagrams at the nodes of a binary tree. Start with the partition $\lambda=(1)$ at the root. For any partition $\mu=(\mu_1, \mu_2, \ldots, \mu_t),$ let its left child be $(\mu_1, \mu_2, \ldots, \mu_t,1)$ and its right child be $(\mu_1+1, \mu_2+1, \ldots, \mu_t+1)$. Then following the unique path from a diagram upward to the root will correspond to the operation described in Remark \ref{remark:operationonpartitions}. This is illustrated in Figure \ref{fig:BranchingDiagram}

\begin{figure}[H]
\centering
\begin{tikzpicture}[scale=1.0]
\node at (0,0){\ytableausetup{centertableaux}
  \ydiagram{1}};

\node at (-4,-2){\ytableausetup{centertableaux}\ydiagram{1,1}};

     \node at (-6,-4){\ytableausetup{centertableaux} \ydiagram{1,1,1}};
            \node at (-7,-6){\ytableausetup{centertableaux} \ydiagram{1,1,1,1}};
            \node at (-5,-6){\ytableausetup{centertableaux} \ydiagram{2,2,2}};
            \node at (-4.3,-5.8){0};
            \node[text=red] at (-4.3,-6.3){0};
            \node[text=red] at (-4.3,-6.8){0};
              \node at (-5.3,-7.3){1};
              \node[text=red] at (-4.8,-7.3){1};

                 \node at (-2,-4){\ytableausetup{centertableaux} \ydiagram{2,2}};
            \node at (-3,-6){\ytableausetup{centertableaux} \ydiagram{2,2,1}};
            \node at (-1,-6){\ytableausetup{centertableaux} \ydiagram{3,3}};

\node at (4,-2){\ytableausetup{centertableaux}\ydiagram{2}};
     \node at (2,-4){\ytableausetup{centertableaux} \ydiagram{2,1}};
            \node at (1,-6){\ytableausetup{centertableaux} \ydiagram{2,1,1}};
            \node at (3,-6){\ytableausetup{centertableaux} \ydiagram{3,2}};

     \node at (6,-4){\ytableausetup{centertableaux} \ydiagram{3}};
            \node at (5,-6){\ytableausetup{centertableaux} \ydiagram{3,1}};
            \node at (7,-6){\ytableausetup{centertableaux} \ydiagram{4}};
   \draw  (0,-0.5) -- (-4,-1.5);
   \draw  (0,-0.5) -- (4,-1.5);
   \draw(-4,-3) -- (-2,-3.5);
     \draw(4,-3) -- (2,-3.5);
     \draw(-4,-3) -- (-6,-3.5);
     \draw(4,-3) -- (6,-3.5);

      \draw(-6,-5.2) -- (-7,-5.4);
      \draw(-6,-5.2) -- (-5,-5.4);

       \draw(-2,-5.2) -- (-3,-5.4);
      \draw(-2,-5.2) -- (-1,-5.4);

       \draw(2,-5.2) -- (1,-5.4);
      \draw(2,-5.2) -- (3,-5.4);
       \draw(6,-5.2) -- (5,-5.4);
      \draw(6,-5.2) -- (7,-5.4);

\end{tikzpicture}
\caption{Partitions on the Calkin-Wilf Tree}
  \label{fig:BranchingDiagram}
\end{figure}
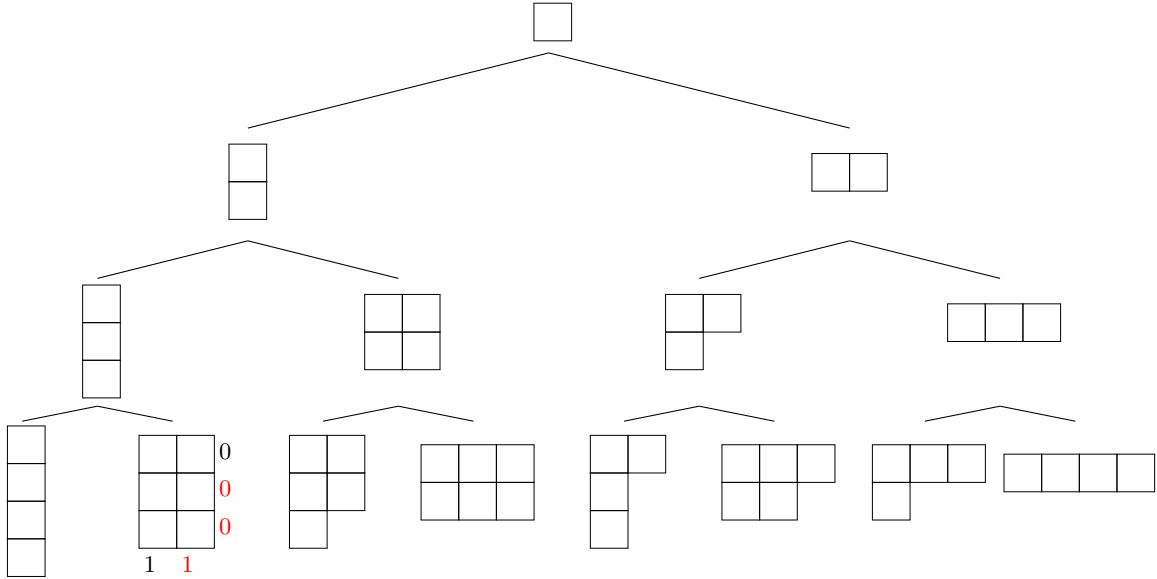

\begin{remark}
\label{remark:youngdiagramoppositeorder}
For the Young diagrams in Figure \ref{fig:BranchingDiagram}, the diagram $[\lambda]$ occupies the same spot as the binary sequence which is $B(\lambda)^r$. This is because we defined $B(\lambda)$ reading from southwest to northeast to match the existing literature. For example if $\lambda=(2,2,2)$ we see that $B(\lambda)=100$ but $\lambda$ occupies the same position as $001$ in the tree. The choice of zero or one for vertical/horizontal moves is arbitrary, as is the choice of which direction to read the sequence off, but the results are all comparable independent of these choices.
\end{remark}

For a partition $\lambda \vdash n$, recall the conjugate or transpose partition $\lambda' \vdash n$ can be obtained by swapping rows and columns in the Young diagram. For a box $(i,j)$ in the Young diagram, the $(i,j)$ hook length is defined as $\lambda_i -i +\lambda'_j-j+1$.

\begin{proposition}
  \label{prop:basicpropertiesofYoungbinarytree}
  \begin{enumerate}
    \item Let $\lambda \vdash n$. Then $B(\lambda')=\overline{B(\lambda)^r}$ is the inverse of the reverse of $B(\lambda)$. Thus, antipalindromic paths correspond to self-conjugate partitions $\lambda=\lambda'$.
    
  \item Given $\lambda$ and $B(\lambda)$, to find the partition $P(B(\lambda)^r)$ corresponding to the reverse sequence of $B(\lambda)$ one can use the algorithm described in Prop. \ref{prop: describereversal}.
    \item Row $t$ of Figure \ref{fig:BranchingDiagram} contains all partitions with first column hook length $h_{11}(\lambda)=t$. This value is also sometimes called the perimeter of the partition.

  \end{enumerate}
\end{proposition}

\begin{remark}
  \label{remark:Klein4action}
  From Proposition \ref{prop:basicpropertiesofYoungbinarytree} we see that we have an action of the Klein four group on the diagram in Figure \ref{fig:BranchingDiagram}. The operations come from the the operations of inverse (I) and reversal (R) on the corresponding sequences. We have seen that $R$ fixes the first row and column while replacing the rest with its complement inside a rectangle, whereas $R\circ I=I \circ R$ takes the transpose of the partition. Thus $I$ does both operations, taking the transpose and then performing the complement operation.
\end{remark}

\section{Palindrome Partitions}
Say a partition $\lambda \vdash n$ is a \emph{palindrome partition} if the sequence $B(\lambda)$ is a palindrome, i.e. $B(\lambda)=B(\lambda)^r$. For example in Figure \ref{fig:codingofpartitions}, we see that $B(55331)=01100110$, so $(5,5,3,1)$ is a palindrome partition. Define $PP(n)$ to be the number of palindrome partitions of $n$. Observe that the partitions $(n)$ and $(1^n)$ are palindromes, corresponding to the sequences of all ones and all zeros respectively, so for $n>1$ we have $PP(n) \geq 2$. Table \ref{tab:palindromesequence} gives the sequence $PP(n)$ up to $n=40$. After this paper appeared in preprint form, this sequence was added to the Online Encyclopedia of Integer Sequences \cite{oeis} as A368548. The values of $n+1$ where $PP(n)=2$ already suggests something interesting is happening!

\begin{table}[H] 
\centering

\caption{Count of Palindrome Partitions}

\begin{tabular}{cc@{\hspace{0.5in}}cc@{\hspace{0.5in}}cc@{\hspace{0.5in}}cc}
\toprule
$n$ & $PP(n)$ & $n$ & $PP(n)$ & $n$ & $PP(n)$ & $n$ & $PP(n)$ \\
\midrule
1 & 1 & 11 & 10 & 21 & 12 & 31 & 38 \\
2 & 2 & 12 & 2 & 22 & 2 & 32 & 34 \\
3 & 2 & 13 & 8 & 23 & 36 & 33 & 18 \\
4 & 2 & 14 & 10 & 24 & 12 & 34 & 46 \\
5 & 4 & 15 & 10 & 25 & 14 & 35 & 104 \\
6 & 2 & 16 & 2 & 26 & 24 & 36 & 2 \\
7 & 4 & 17 & 18 & 27 & 36 & 37 & 20 \\
8 & 4 & 18 & 2 & 28 & 2 & 38 & 46 \\
9 & 6 & 19 & 20 & 29 & 60 & 39 & 108 \\
10 & 2 & 20 & 16 & 30 & 2 & 40 & 2 \\
\bottomrule
\end{tabular}

\label{tab:palindromesequence} 
\end{table}

We will use the description in Proposition \ref{prop: describereversal} to compute a generating function $\sum_{n=0}^\infty PP(n)q^n$ for $PP(n)$. We will sum over all possible choices for $\lambda_1$ and $\lambda'_1$. Let ``Case 1" be when both both $\lambda_1$ and $\lambda'_1$ are odd, so $A$ and $B$ are both even. Let $A=2k$ and $B=2l$. We must count partitions $\mu$ in the $2k \times 2l$ rectangle so that $\mu$ is equal to its complement inside the rectangle. Dividing the large rectangle into 4 rectangles of size $k \times l$, it is clear that the upper left rectangle must be included in $\mu$ and the lower right rectangle must be empty. We can select an arbitrary partition $\tau$ for lower left rectangle, and then its complement in the $k \times l$ rectangle for the upper right, as shown on the left side of Figure \ref{fig:forgeneratingfunction}. We have $\mu$ occupying $2kl$ boxes. It is well know that there are $\binom{k+l}{k}$ choices for $\tau$ fitting inside a $k \times l$ rectangle. We also have $2k+2l+1$ boxes in the first row and column. Summing over all possible $k$ and $l$ we have a contribution to the generating function of:
\begin{equation}\label{eq:evenevencase}
  \sum_{k=0}^{\infty}\sum_{l=0}^{\infty}\binom{k+l}{k}q^{2kl+2k+2l+1}.
\end{equation}

Now suppose $\lambda_1$ is odd and $\lambda'_1$ is even, call this ``Case 2a." So we are considering $\mu$ inside a $2k+1 \times 2l$ box as shown on the right side of Figure \ref{fig:forgeneratingfunction}. It is clear that the shaded $1 \times l$ rectangle must be entirely included in $\mu$. If it were not there would be a shaded box in the $1 \times l$ rectangle to its right with an unshaded box to its left, and $\mu$ would not be a valid palindrome partition. So in this case we have $\mu$ filling up $(2k+1)l$ boxes,  $\binom{k+l}{k}$ choices for $\tau$ and $2k+2l+2$ boxes in the first row and column. Summing over all possible $k$ and $l$ we have a contribution to the generating function of:
\begin{equation}\label{eq:evenoddcase}
  \sum_{k=0}^{\infty}\sum_{l=0}^{\infty}\binom{k+l}{k}q^{2kl+2k+3l+2}.
\end{equation}

Now ``Case 2b" is when $\lambda_1$ is even and $\lambda'_1$ is odd. This would simply give the transpose of the diagram on the right of Figure \ref{fig:forgeneratingfunction}, and gives another contribution equal to \eqref{eq:evenoddcase}, as the sum is symmetric in $k$ and $l$. 

There is no ``Case 3" where both $\lambda_1$ and $\lambda'_1$ are even, as there can be no palindrome partitions as $\mu$ would need to fill half of the boxes in an odd by odd rectangle. The three cases are summarized in Table \ref{table:paritycases}.
{\renewcommand{\arraystretch}{1.2}
\begin{table}[H]
  \centering
  \caption{Possible parities for $\lambda_1$ and $\lambda'_1$}
  
\begin{tabular}{c@{\hspace{0.5in}}c@{\hspace{0.5in}}cc@{\hspace{0.5in}}c}
  \hline
  Case & $A= \lambda'_1-1=\#$ 0's & $B= \lambda_1-1=\#$ 1's & n & 2(n+1) \\
    \hline
  1 & 2k &2l  & 2k+2l+2kl+1 &4(k+1)(l+1) \\
  2a. & 2k+1 & 2l & 2k+3l+2kl+2&(2l+2)(2k+3) \\
    2b. & 2k & 2l+1 &3k+2l+2kl+2 &(2k+2)(2l+3) \\
  \hline
\end{tabular}

  \label{table:paritycases}
  \end{table}}

Thus we have proven:

\begin{theorem}
  \label{thm:GFforPP(n)}
We have the following generating function:
$$\sum_{n=0}^{\infty}PP(n)q^n=\sum_{k=0}^{\infty}\sum_{l=0}^{\infty}\binom{k+l}{k}q^{2kl+2k+2l+1} +   2\sum_{k=0}^{\infty}\sum_{l=0}^{\infty}\binom{k+l}{k}q^{2kl+2k+3l+2}.$$
\end{theorem}

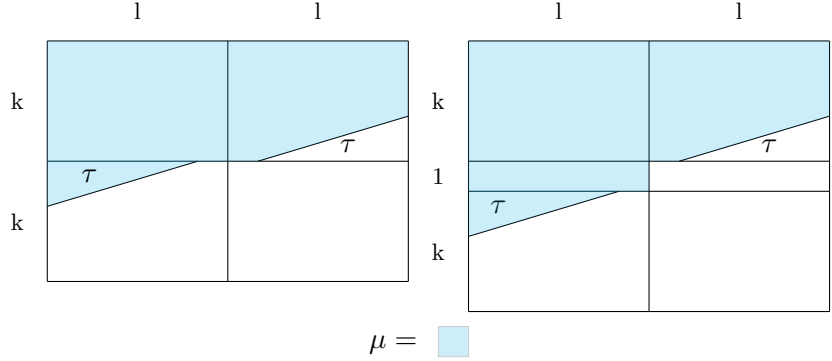
\begin{figure}[H]
\centering
   \begin{tikzpicture}[scale=0.4]
         \draw  (0,0) -- (0,8);
         \draw  (0,8) -- (12,8);
        \draw  (12,8) -- (12,0);
          \draw  (0,0) -- (12,0);
          \draw  (0,4) -- (12,4);
        \draw  (6,0) -- (6,8);
           \node at (-1,2){k};
            \node at (-1,6){k};
             \node at (3,9){l};
            \node at (9,9){l};
               \draw (0,2.5) -- (5,4);
                 \draw (7,4) -- (12,5.5);

  \fill[cyan, opacity=0.2] (0,4) rectangle (6,8);
    \fill[cyan, opacity=0.2] (0,2.5) --(0,4)--(5,4);
        \fill[cyan, opacity=0.2](7,4)--(6,4)--(6,8)--(12,8) --(12,5.5);

   \node at(11.5,-2){\large $\mu=$};
  
  \draw[fill=cyan, opacity=0.2](13,-1.5) rectangle (14,-2.5);

         \draw  (14,-1) -- (14,8);
         \draw  (14,8) -- (26,8);
        \draw  (26,8) -- (26,-1);
          \draw  (14,-1) -- (26,-1);
          \draw  (14,3) -- (26,3);
        \draw  (20,-1) -- (20,8);
           \node at (13,1){k};
            \node at (13,6){k};
            \node at (15,2.5){\large $\tau$};
            \node at (1.4,3.5){\large $\tau$};
             \node at (24,4.5){\large $\tau$};
          \node at (10,4.5){\large $\tau$};
             \node at (17,9){l};
            \node at (23,9){l};
               \draw (14,1.5) -- (19,3);
                 \draw (21,4) -- (26,5.5);
                  \draw  (14,4) -- (26,4);

 \node at (13,3.5){1};
  \fill[cyan, opacity=0.2] (14,3) rectangle (20,8);
    \fill[cyan, opacity=0.2] (14,1.5) --(14,3)--(19,3);
        \fill[cyan, opacity=0.2](21,4)--(20,4)--(20,8)--(26,8) --(26,5.5);
    \end{tikzpicture}
\caption{Counting Palindrome Partitions}
  \label{fig:forgeneratingfunction}
\end{figure}

We are grateful to James Sellers for pointing out a simplification to the generating function in Theorem \ref{thm:GFforPP(n)} using the binomial series $\frac{1}{(1-x)^{n+1}} = \sum_{j=0}^{\infty} \binom{n+j}{j}x^j$. In the first double sum we can pull out $q^{2k+1}$ term and set $x=q^{2k+2}$. In the second, we pull out $q^{2k+2}$ and set $x=q^{2k+3}$, to obtain:
\begin{corollary}

\begin{eqnarray*}\sum_{n=0}^{\infty}PP(n)q^n &=&\sum_{k=0}^{\infty}\left[ \frac{q^{2k+1}}{(1-q^{2k+2})^{k+1}}  + 2\frac{q^{2k+2}}{(1-q^{2k+3})^{k+1}} \right ]\\
&=&\sum_{k=0}^{\infty} \left[ q^k \Big(\frac{q}{1-q^{2k+2}}\Big)^{k+1} + 2\Big(\frac{q^{2}}{1-q^{2k+3}}\Big)^{k+1}\right].
\end{eqnarray*}

\end{corollary}

In Table \ref{table:paritycases} we added a fourth column showing the factorization of $2(n+1)$. We can now use this to explain why $PP(n)=2$ occurs only when $n=3$ or $n+1$ is prime.

\begin{theorem}
  \label{thm: PrimesinPP(n)}The number $PP(n)=2$ if and only if $n=3$ or $n+1>2$ is prime.
  \end{theorem}

\begin{proof}

Consider the factorizations of $n+1$ in Table \ref{table:paritycases}. If $n+1$ is prime, we are either in Case 2a with $l=0$ or Case 2b with $k=0$, corresponding to the partitions $(1^n)$ and $(n)$ respectively. In this case there are no other palindrome partitions. Now suppose $n+1$ is not prime, so let $n+1=xy$ with $x, y>1$. We will explain how to find additional palindrome partitions of $n$.

If $n+1$ is odd, then so are $x$ and $y$. We can find $\lambda$ in Case 2a with $l=x-1\geq 2$ and $k=(y-3)/2\geq 0$. Then any of the $\binom{k+l+1}{k}$ choices for $\tau$ will produce a palindrome partition  $\lambda \vdash n$ not equal to $(n)$ or $(1^n)$. A corresponding choice in Case 2b will give the conjugate $\lambda'$. In the smallest example where $x=y=3$, we can choose $k=0$, $l=2$ or $k=2, l=0$ and we obtain  $\lambda=(5,3)$ and its conjugate $(2,2,2,1,1)$.

Finally suppose $n+1=2x$. If $x>2$ is even, we can choose $k=1$, $l=x/2-1$ in Case 1, to obtain a palindrome partition $\lambda \vdash n$ not equal to $(n)$ or $(1^n)$. If $x \geq 3$ is odd, we can choose $l=1$ and $k=(x-3)/2$ in Case 2a to obtain a palindrome partition $\lambda \vdash n$ not equal to $(n)$ or $(1^n)$.

This leaves only the special case where $n+1=4$, where we see that $n+1$ is not prime but the only palindrome partitions of $n=3$ are $(3)$ and $(1,1,1)$.
\end{proof}

\begin{remark}
  \label{remark: sequenceofnwithPP(n)=2}
  The sequence of integers $n$ with $PP(n)=2$, given by $\{2,3,4,6,10,12,16,18,22,\ldots\}$, is A068499 at oeis.org : ``Numbers $n$ such that $m!$ reduced modulo $(n+1)$ is not zero." Of course this differs in only two terms from the sequence of $n$ with $n+1$ prime.
\end{remark}

\section{Generating Palindrome Partitions of $n$}
We have seen that, given $n$, there are always two palindrome partitions of $n$ corresponding to sequences $B(\lambda)$ of length $n-1$, namely the all zero sequence corresponding to $(1^n)$ and the all one sequence corresponding to $(n)$. If we want to generate all palindrome partitions of $n$, we need to know what other sequence lengths can occur, and for each length, how many zeros and ones are included. This will specify the first part $\lambda_1$ and the first column $\lambda'_1$, and values for $k$ and $l$. We simply choose all possible $\mu$ as in the proof of Theorem \ref{thm:GFforPP(n)}, to obtain a list of all palindrome partitions of $n$.
{\renewcommand{\arraystretch}{1.3}
\begin{table}[H]
  \centering
  \caption{Palindrome partitions of $n=11$}
\begin{tabular}{ccccccccccc}
  \hline
  $\lambda$& 11 & 7,4 & 5,5,1 & 5,4,2 & 5,3,3 & $3^3,1^2$ & $3^2,2^2,1$ & $3,2^4$ & $2^4,1^3$ & $1^{11}$ \\\hline
  $B(\lambda)$ & $1^{10}$ & $1^301^3$ & $01^4 0$ & 101101 & 110011 & 001100 & 010010 & $10^41$ & $0^310^3$ & $0^{10}$\\
  $l(B(\lambda))$ & 10 & 7 & 6 & 6 & 6 & 6 & 6 & 6 & 7 & 10 \\
  \hline
\end{tabular}

  \label{table:palindromesofn=11}
  \end{table}
}

\begin{definition}
  \label{def:palindromelength}
  Let $PL(n)$ be the number of lengths among all sequences $B(\lambda)$ as $\lambda$ runs over all palindrome partitions of $n$. This is also the number of distinct perimeters among all palindrome partitions of $n$.
\end{definition}

So in Table \ref{table:palindromesofn=11}, we see that $PL(11)=3$ corresponding to the lengths 6,7, 10. For a given length $m$, we have $m=A+B$, where the sequence has $A$ zeros and $B$ ones. Recall that swapping zeros and ones gives a sequence also corresponding to a palindrome partition of $n$. We now classify precisely which lengths occur and prove that, for each possible length, the split into zeros and ones is unique (up to swapping zeros and ones).

\begin{theorem}
\label{thm: palindrome lengths}
$PL(n)$ is the number of factorizations $xy=2(n+1)$ where $0< x \leq y \leq n$. This is the sequence A211270 in \cite{oeis}, shifted by one. Moreover, suppose there is a palindrome partition $\lambda \vdash n$ with $B(\lambda)$ having length $m=A+B$ with $A$ zeros and $B$ ones. Then any other palindrome partition $\mu \vdash n$ with $B(\mu)$ of length $m$ must have $A$ zeros and $B$ ones or $B$ zeros and $A$ ones.

\end{theorem}

\begin{proof}
   Suppose we have a palindrome partition $\lambda$ of $n$ where $B(\lambda)$ has length $m=A+B$ as in Table \ref{table:paritycases}. Thus $k$ and $l$ are determined and we obtain a factorization as desired from the final column. 
   
   Conversely, suppose we have such a factorization $xy=2(n+1)$. We prove there is a unique corresponding length $m$ and decomposition $m=A+B$. It is the case that we have one odd factor or no odd factor.
    If one of the factors is odd, say $x$, then $x= 2l+3$ and $y = 2k+2$, as in case $2$. Then it follows that $A = 2k = y-2$ and $B = 2l = x-3$.
    If neither factor is odd, then $x=2l+2$ and $y = 2k+2$ which is case $1$. So $A=x-2$ and $B=y-2$, and we have a palindrome partition with $B(\lambda)$ of length $A+B$.
    Suppose that for a factorization of $2(n+1)$ we have that both $x$ and $y$ are even where $m = 2k+2l$ and also $m=2\tilde{k}+2\tilde{l}$ for some $k,l, \tilde{k}, \tilde{l}$ where $k \neq \tilde{k}$ and $l \neq \tilde{l}$. It could be the case that $k = \tilde{l}$ and $l = \tilde{k}$, which would be interchanging the zeros and ones. Consider a nontrivial solution where $k+l = \tilde{k} + \tilde{l}$. Then $2(n+1) = (2k+2)(2l+2)=(2\tilde{k}+2)(2\tilde{l}+2)$ which simplifies to $kl=\tilde{k} \tilde{l}$. Then 
    \begin{align*}
        kl & = \tilde{k}\tilde{l} \\
         kl + l^2 & = \tilde{k}\tilde{l} +l^2 \\
         l(k+l) & = \tilde{k}\tilde{l} +l^2 \\
         l(\tilde{k}+\tilde{l}) & = \tilde{k}\tilde{l} +l^2 \\
         (l-\tilde{l})(\tilde{k}-l) &= 0.
    \end{align*}
    Thus $l = \tilde{l}$ or $\tilde{k} = l$. Now suppose for a factorization of $2(n+1)$ that $x$ is odd without loss of generality, which would mean $m = 2k+2l+1$ and $m=2\tilde{k}+2\tilde{l}+1$ for some $k,l, \tilde{k}, \tilde{l}$ where $k \neq \tilde{k}$ and $l \neq \tilde{l}$. Again, the case that $k = \tilde{l}$ and $l = \tilde{k}$, which would be interchanging the zeros and ones. Suppose we have a nontrivial solution where $k+l = \tilde{k} + \tilde{l}$. Then $2(n+1) = (2k+2)(2l+3)=(2\tilde{k}+2)(2\tilde{l}+3)$ which simplifies to $2kl+k=2 \tilde{k} \tilde{l} + \tilde{k}$. Then 
    \begin{align*}
        2kl+k & = 2\tilde{k}\tilde{l} + \tilde{k}\\
         2kl+k-k+2k^2 & = 2\tilde{k}\tilde{l}+\tilde{k}-k+2k^2 \\
         2k(l+k) &= 2\tilde{k}\tilde{l}+\tilde{k}-k+2k^2 \\
         2k(\tilde{l}+\tilde{k}) &= 2\tilde{k}\tilde{l}+\tilde{k}-k+2k^2 \\
         (\tilde{k}-k)(2k-1-2\tilde{l}) &= 0. \\
    \end{align*}
    Thus $k = \tilde{k}$ or $2k = 2\tilde{l}+1$.
    Therefore, it's the case that for each word length of $m$, there is a unique number of zeros and ones, up to swapping zeros and ones.
\end{proof}

\begin{example}
\label{ex:pp29}
We demonstrate how to use the results above to efficiently generate all $PP(29)=60$, palindrome partitions of $n=29$. There are five relevant factorizations of $2(n+1)=30$. For each, we can compute $k,l,A,B$ as in the proof above, using Table \ref{table:paritycases}. Once we have these, the first row $\lambda_1$ and column $\lambda'_1$ are fixed. Then we have $\binom{k+l}{l}$ choices for the partition $\tau$ as in Figure \ref{fig:forgeneratingfunction}. Table \ref{table:generatepalindromeof29} shows the results, where the entries in the penultimate column sum to 60 as expected.

\end{example}

\begin{table}[H]
  \centering
  \caption{Generating Palindrome Partitions of 29}
 {\renewcommand{\arraystretch}{1.3} 
\begin{tabular}{c@{\hspace{0.2in}}c@{\hspace{0.2in}}c@{\hspace{0.2in}}c@{\hspace{0.2in}}c@{\hspace{0.2in}}c@{\hspace{0.2in}}c}
  \toprule
 $60=x\cdot y$&l & k &\# of zeros& \# of ones& \# of $\lambda$& Example \\ \midrule
 $2\cdot 30$ & 0 & 14& 0&28 & 1& $(29)$\\ 
  $2\cdot 30$ & 14 & 0& 28&0 & 1&$(1^{29})$\\ 
    $3\cdot 20$ & 9 & 0& 1&18 & 1&$(19,10)$\\ 
      $3\cdot 20$ & 0 & 9& 18&1 & 1&$(2^{10},1^9)$\\ 
           $4\cdot 15$ & 6 & 1& 2&13 & $\binom{7}{1}$&$(14,14,1)$\\
                $4\cdot 15$ & 1 & 6& 13&2 & $\binom{7}{1}$&$(3,2^{13})$\\
                       $5\cdot 12$ & 5 & 1& 3&10 & $\binom{6}{1}$&$(11,11,6,1)$\\
                        $5\cdot 12$ & 1 & 5& 10&3 & $\binom{6}{1}$&$(4,3^5, 2^5)$\\
                         $6\cdot 10$ & 4 & 2& 4&8 & $\binom{6}{2}$&$(9^3,1,1)$\\
                         $6\cdot 10$ & 2 & 4& 8&4 & $\binom{6}{2}$&$(5^5,1^4)$\\

  \bottomrule
\end{tabular}}

  \label{table:generatepalindromeof29}
  \end{table}

\begin{remark}
    In Example \ref{ex:pp29} we generate palindrome partitions of $n$ sorted by divisors of $2n+1$. Chai Wah Wu has used this idea to give an expression for $PP(n)$, which has the flavor of a divisor sum, and can be found in the OEIS entry for $PP(n)$.
\end{remark}

\section{Partitions with weight fixed by reversal}
As we mentioned in Remark \ref{remark:Klein4action}, sequence reversal gives an action of order two on all partitions given by $\lambda \rightarrow P(B(\lambda)^r)$. The palindrome partitions are fixed by this operation. In this section, we consider partitions where $\lambda$ and $P(B(\lambda)^r)$ are both partitions of the same integer. In other words, these partitions  remain in the same row of Figure \ref{fig:BranchingDiagram} under this action. Let $R(n)$ be the number of partitions $\lambda \vdash n$ where $P(B(\lambda)^r) \vdash n$.

Computing $R(n)$ is similiar to counting $PP(n)$, but in Figure \ref{fig:forgeneratingfunction} we no longer have the requirement that the partition be equal to the rotation of its complement. Instead we require that it fill exactly half of the boxes in the $2k \times 2l$ or $(2k +1) \times 2l$ rectangle. (As with the palindrome case, if the rectangle has an odd number of boxes then there are no partitions to count).

Recall that the q-binomial coefficient ${n \brack k}_q$ is the generating function for partitions that fit inside an $(n-k) \times k$ rectangle, and is called a Gaussian polynomial. The coefficient of $q^m$ counts partitions of $m$ fitting inside a $(n-k) \times k$ rectangle, so it has leading term $q^{k(n-k)}$. In place of the binomial coefficients in Theorem \ref{thm:GFforPP(n)}, we will need the coefficient of the middle degree term in this polynomial. There is a nice description of this coefficient in the case either $k$ or $n-k$ is even, which we prove here.

\begin{definition}
\label{def:Tnk}
Let $T(n,k)$ be the number of nondecreasing sequences of length n, with integer entries in $[-k,k]$, summing to zero.
\end{definition}
For example if $n=5$ and $k=4$ a possible sequence would be $\{-4,-2,1,1,4\}$. The array $T(n,k)$ is given as sequence A183917 in the OEIS.

\begin{proposition}
\label{prop: middle coef q binomial} The coefficient of $q^{kl}$ in the Gaussian polynomial ${2k+l \brack l}_q$ is $T(l,k)$.
\end{proposition}

\begin{proof}
    The coefficient in question counts all partitions of $kl$ which fit inside an $l \times 2k$ rectangle, i.e. have at most $l$ parts, each less than or equal to $2k$. We give a bijection to the sequences counted by  $T(l,k)$. Given such a partition $\lambda=(\lambda_1, \lambda_2, \ldots, \lambda_l) \vdash kl$, the corresponding sequence is $\{\lambda_l-k, \lambda_{l-1}-k, \lambda_{l-2}-k, \ldots, \lambda_1-k\}$.  For example the sequence above corresponds to the partition $(8,5,5,2,0)$ inside the $5 \times 8$ rectangle.
\end{proof}

We can now write the generating function for $R(n)$ as in Theorem \ref{thm:GFforPP(n)}. The first term is from the $2k \times 2l$ rectangle so we get $T(2l,k)$. The second term is the $(2k+1) \times 2l$ rectangle, which has the same count as a $2l \times (2k+1)$ rectangle, namely $T(2k+1,l)$. Thus: 

\begin{theorem}
  \label{thm:GFforR(n)}
We have the following generating function:
$$\sum_{n=0}^{\infty}R(n)q^n=\sum_{k=0}^{\infty}\sum_{l=0}^{\infty}T(2l,k)q^{2kl+2k+2l+1} +   2\sum_{k=0}^{\infty}\sum_{l=0}^{\infty}T(2k+1,l)q^{2kl+2k+3l+2}.$$
\end{theorem}

\begin{table}[H] 
\centering

\caption{Partitions with weight fixed by reversal}

\begin{tabular}{cc@{\hspace{0.5in}}cc@{\hspace{0.5in}}cc@{\hspace{0.5in}}cc}
\toprule
$n$ & $R(n)$ & $n$ & $R(n)$ & $n$ & $R(n)$ & $n$ & $R(n)$ \\
\midrule
1 & 1 & 11 & 10 & 21 & 12 & 31 & 76 \\
2 & 2 & 12 & 2 & 22 & 2 & 32 & 90 \\
3 & 2 & 13 & 8 & 23 & 52 & 33 & 18 \\
4 & 2 & 14 & 14 & 24 & 28 & 34 & 198 \\
5 & 4 & 15 & 10 & 25 & 14 & 35 & 320 \\
6 & 2 & 16 & 2 & 26 & 52 & 36 & 2 \\
7 & 4 & 17 & 20 & 27 & 80 & 37 & 20 \\
8 & 4 & 18 & 2 & 28 & 2 & 38 & 142 \\
9 & 6 & 19 & 28 & 29 & 120 & 39 & 388 \\
10 & 2 & 20 & 28 & 30 & 2 & 40 & 2 \\
\bottomrule
\end{tabular}

\label{tab:RPsequence} 
\end{table}

Of course $PP(n) \leq R(n)$. Comparing Tables \ref{tab:palindromesequence} and \ref{tab:RPsequence}, we see the smallest $n$ with $PP(n)<R(n)$ is $n=14$, where $PP(14)=10$ and $R(14)=14$. The four extra partitions which are not palindromes are $\lambda=(44222), R(\lambda)=(43331)$ and $\mu=(5522), R(\mu)=(5441)$.

\begin{remark}
    The proof of Theorem \ref{thm: PrimesinPP(n)} applies to the more general setting of computing $R(n)$, i.e. $R(n)=2$ precisely when $P(n)=2$.
\end{remark}

\section{Problems}

We suggest some problems for future research. 
\begin{problem}
    Can we write the generating functions in Theorem \ref{thm:GFforPP(n)} and \ref{thm:GFforR(n)} in a more compact form $\sum f(n)q^n$? Can we see clearly from the generating function why $f(n)=2$ when $n+1$ is prime? 
\end{problem}

\begin{problem}
    The traditional Young's lattice has all partitions of $n$ in row $n$ with edges corresponding to removing and/or adding a single box. Classically this describes the branching of irreducible representations of the symmetric group. Does Figure \ref{fig:BranchingDiagram} have any representation- theoretic interpretation?
\end{problem}

\begin{problem}
    Proposition \ref{prop: describereversal} gives a quick way to identify palindrome partitions $\lambda$ without calculating $B(\lambda)$. Can one do the same for fractions in the Calkin-Wilf tree corresponding to palindrome sequences, i.e. determine if a fraction corresponds to a palindrome without doing the continued fraction expansion? Can we describe what the reversal operation does to fractions?
\end{problem}

\section{Acknowledgements}
The authors would like to thank Matthew Just and Robert Schneider for suggesting looking at this operation on partitions and William Keith for help with generating functions.

\bibliography{PalindromesReferences}
\end{document}